\documentclass{amsart}

\usepackage{vmargin,amsmath,amsfonts,amssymb,amsthm,xcolor,amscd,latexsym,xspace,enumerate}
\usepackage[pagebackref, colorlinks=true,linkcolor=blue,citecolor=red]{hyperref}
\usepackage{hyperref}
\usepackage{pdfpages}
\usepackage{tikz-cd}
\usetikzlibrary{arrows}
\usetikzlibrary{intersections}
\usepackage{pgfplots}
\usetikzlibrary{calc,3d,shapes, pgfplots.external, intersections}
\usepackage{pgfplots}
\pgfplotsset{compat=1.18}
\usepackage{multicol}

\newtheorem{theorem}{Theorem}[section]
\newtheorem{corollary}[theorem]{Corollary}
\newtheorem{lemma}[theorem]{Lemma}
\newtheorem{proposition}[theorem]{Proposition}

\theoremstyle{definition}
\newtheorem{definition}[theorem]{Definition}
\newtheorem{remark}[theorem]{Remark}

\newtheorem{example}[theorem]{Example}
\numberwithin{equation}{section}

\begin{document}

\title[On the number of modular pairs...]{On the number of modular pairs\\ 
in finite dimensional Lie algebras on finite fields}

\author[S.K. Muhie]{Seid Kassaw Muhie}
\address{Seid Kassaw Muhie \endgraf  
Institute of Data Science and Digital Technologies\endgraf 
Vilnius University -- Lithuania\endgraf
}

\author[D.E. Otera]{Daniele Ettore Otera}
\address{Daniele Ettore Otera \endgraf  
Institute of Data Science and Digital Technologies \endgraf  Vilnius University -- Lithuania \endgraf
}

\author[F.G. Russo]{Francesco G. Russo}
\address{Francesco G. Russo \endgraf
School of Science and Technology \endgraf 
University of Camerino -- Italy \endgraf
and \endgraf
Department of Mathematics and Applied Mathematics \endgraf University of the Western Cape -- South Africa\endgraf
and
\endgraf
Department of Mathematics and Applied Mathematics \endgraf University of Cape Town -- South Africa
}



\begin{abstract} 
Given a finite dimensional Lie algebra $L$ on a finite field $\mathbb{F}_{p^n}$ of prime power order $p^n$ (with $n$ positive integer and $p$ prime), we consider the number of modular pairs $(A,B)$ in the lattice of all subalgebras $\mathcal{L}(L)$ and introduce the notion of ``subalgebra  commutativity degree'' of $L$. This represents the probability to find that two randomly chosen  subalgebras $A$ and $B$ of $L$ are permutable. We investigate the subalgebra commutativity degree of $L$ in connection with recent techniques of algebraic combinatorics and number theory, providing upper and lower bounds which may influence the structure of $L$.  A specific study for the subalgebra commutativity degree of Heisenberg  algebras  is executed.
\end{abstract}

\maketitle

\section{Introduction and statement of the main results}
In the present paper we denote a  field with the symbol $\mathbb{F}$ and with $\mathrm{char}(\mathbb{F})$ its characteristic. Of course, either    $\mathrm{char}(\mathbb{F})=0$, or $\mathrm{char}(\mathbb{F})=p$ is a prime, and  for finite fields of prime power order $p^n$ we write  $\mathbb{F}_{p^n}$, where $n$ is a positive integer.  A significant number of works, which have been published between the 60s and the 80s, investigated   the structure of a Lie algebra $L$ from restrictions on the lattice $\mathcal{L}(L)$ of its subalgebras (see \cite{amayo1976, amsc,  elduque1, gein, kolman1965, lashi, david1981}). Given  $A, B$  subalgebras of a Lie algebra $L$ on $\mathbb{F}$, Towers and others \cite{elduque1, elduque2, david1981, david1987} called  $(A, B)$ a \textit{modular pair}  in  $\mathcal{L}(L)$ of $L$,  if
\[ \begin{cases}
(A\cup B)\cap C =A\cup (B\cap C) \ \ \mbox{for all subalgebras} \ \ A\subseteq C\\
\\
(A\cup B)\cap C =B\cup (A\cap C) \ \ \mbox{for all subalgebras} \ \ B\subseteq C,\\
\end{cases} \] 
 where  $A\cup B=\langle A,B \rangle $ denotes the subalgebra of $L$ generated by $A$ and $B$, that is, the smallest subalgebra of $L$ containing both $A$ and $B$. Again Towers \cite{david1981} defines  $A$ and $B$ \textit{permutable}, if $$A\cup B= A+B,$$ that is, if $A\cup B$ agrees with the vector space sum of $A$ and $B$, namely $$A+B=\{f_1a+f_2b \mid f_1, f_2 \in \mathbb{F}, a \in A, b\in B\}.$$  From \cite[Lemma 1.1]{david1981},  if  $A$ permutes with $B$, then $(A, B)$ is a modular pair, but  the converse is false. This means that the following implications hold:
   \[A \ \mbox{and} \ B \  \mbox{are ideals in } \ L  \ \Longrightarrow \ A \ 
 \mbox{and} \ B \ \mbox{are permutable in} \ L \ \Longrightarrow \ (A,B) \ \mbox{is a modular pair}.\]
It is possible to produce counterexamples, in order to check that the converse of the above  implications are false, see \cite{elduque1, lashi, david1981}. Lashi and Towers \cite{lashi, david1981} classified \textit{modular Lie algebras}, that is, Lie algebras in which each pair of subalgebras is permutable, that is, Lie algebras $L$ such that  $\mathcal{L}(L)$ satisfies the modular law. These are structurally close to abelian Lie algebras.
Note also that Amayo \cite{amayo1976} shows that for  the subspaces $A$ and $B$  of a Lie algebra $L$ on a field $\mathbb{F}$,   
\[A \ 
 \mbox{and} \ B \ \mbox{are permutable in} \ L \Longleftrightarrow \  [A,B]\subseteq A+B, \ \ \mbox{where}\ \ [A,B]=\langle [a,b] \mid a\in A, b\in B\rangle.\]
 Here $[a,b]=ab-ba$ denotes the Lie bracket on $a \in A$ and $b\in B$ and $[A,B]$  the smallest Lie algebra in $L$ containing the elements of the form $ab-ba$.
Note that a \textit{quasi-ideal} $Q$ of $L$ is a modular subalgebra of  $L$. In particular,  the subalgebra $Q$ of $L$ is a quasi-ideal if it is permutable with every subspace $R$ of $L$, i.e., if $[Q, R]\subseteq Q+R$ for every subspace $R$ of $ L$. This follows the classical terminology of Amayo \cite{amayo1976,amayo1976II}.

On the other hand,  the number of modular pairs in finite dimensional Lie algebras was studied  by Towers \cite{david1981} with different methods, since he found analogies with the ``modularity conditions'' of finite groups, see \cite{rs}. Just to give an idea: generalizing a classical result of Ito \cite{rs} for which a finite group $G=HK$ with $H$ and $K$ abelian subgroups of $G$ must be metabelian,  Kegel \cite{kegel1} proved that  $G=HK$ with $H$ and $K$ nilpotent subgroup of $G$ must be solvable. This opened a long line of research on the ``permutability conditions'' and ``modularity conditions'' for prescribed families of subgroups, see \cite{hhr, lenston, rs}. In particular, Kegel \cite{kegel2} explored the first connections  between subnormal subgroups and permutable subgroups of finite groups. In case of finite dimensional Lie algebras $L$ on  $\mathbb{F}$, Kolman \cite{kolman1965} noted that a corresponding formulation of the result of Kegel was possible, namely if
 $L$ is a finite dimensional Lie algebra such that $L = N_1+N_2$, and $N_1$ and $N_2$ are abelian subalgebras, then   $L$ is metabelian.

It should be said that some difficulties appear when we leave the context of finite groups. For instance, the sum of two randomly chosen subspaces of a vector space $V$ (which has trivially the structure of Lie algebra) is  a vector subspace, but the union of two vector subspaces is not necessarily  a subspace of $V$. On the other hand,  if every element of a finite dimensional Lie algebra $L$ generates an ideal, then $L$ is abelian.  Note also that  one dimensional Lie algebras are abelian and modular.  In fact,  one  and  two dimensional Lie algebras are  solvable and modular, since they haven't  sublattices isomorphic to the pentagonal sublattice $N_5$, see \cite{lenston, rs}. However there are two dimensional Lie algebras which are non-nilpotent.  In analogy with the \textit{subgroup commutativity degree} of finite groups in  \cite{farrokhi2, sr1, sr2, or,  mt}, we may introduce the following notion:

\begin{definition}[Subalgebra commutativity degree]\label{sdl}Given a finite dimensional Lie algebra $L$ on the finite field $\mathbb{F}_{p^n}$ we call  \textit{subalgebra  commutativity degree} of $L$  the quantity 
\[
\mathrm{sd}(L)  =\frac{1}{|\mathcal{L}(L)|^2}\left|\left\{(A, B) \in \mathcal{L}(L) \times  \mathcal{L}(L)\  \mid \ \ [A, \  B] \subseteq A+B   \right\}\right|. 
\]
\end{definition} 

It is useful to note that  Definition \ref{sdl} measures the  probability that two randomly chosen subalgebras $A$ and $B$ of $L$ on a finite field permute, where $|\mathcal{L}(L)|$  turns out to be finite.   If  $L= \mathfrak{sl}_2(\mathbb{C})$ is the special linear Lie algebra on the complex field $\mathbb{C}$,  or  if $L$ is an infinite dimensional Lie algebra of smooth vector fields on a manifold, then $|\mathcal{L}(L)|$ would be infinite. Of course, according to Definition \ref{sdl}, $L$ should be finite dimensional on a finite field, so these two cases are automatically excluded. On the other hand, if  $m$ is a positive integer and consider the \textit{Heisenberg algebra}  $\mathfrak{h}(m)$ on $\mathbb{F}_p$ of dimension $\mathrm{dim} \ \mathfrak{h}(m)=2m+1$, namely the Lie algebra
\[\mathfrak{h}(m) = \langle x_i, y_i, z \mid [x_i, y_i] = z,  \ 1 \leq i \leq m\rangle,\]
then it is reasonable to ask the values of $\mathrm{sd}(\mathfrak{h}(m))$. In fact $\mathfrak{h}(m)$  is a nilpotent finite dimensional Lie algebra of nilpotency class 2 such that \[|\mathfrak{h}(m)|=p^{\mathrm{dim} \ \mathfrak{h}(m)}, \ \ [\mathfrak{h}(m), \mathfrak{h}(m)]=\langle z\rangle  \ \ \mbox{and} \ \ [[\mathfrak{h}(m), \mathfrak{h}(m)], \mathfrak{h}(m)] \  \mbox{is trivial}.\] The Heisenberg algebra, $\mathfrak{h}(m)$ on  $\mathbb{F}_p$ plays a critical role in various fields, including quantum mechanics, cryptography, and digital signal processing.  
On the basis of the available theory of the subgroup commutativity degree in \cite{farrokhi1, farrokhi2,   sr1, sr2, or, mt} for finite groups, we are able to develop a body of results (as per Definition \ref{sdl}) which are not always similar to the corresponding case of finite groups.  
Our first main result deals with a computation of the subalgebra commutativity degree for certain families of Heisenberg algebras.
\begin{theorem} \label{heisenberg}
   For  \(\mathfrak{h}(1)=\langle x,y,z \mid [x,y]=z\rangle \) on \(\mathbb{F}_{p}\) we have 
   $$\mathrm{sd}(\mathfrak{h}(1))=\frac{3p^3+12p^2+16p+16}{{(p^2+2p+4)^2}}.$$
\end{theorem}

The \textit{special linear Lie algebra} $\mathfrak{sl}_2(\mathbb{F}_q)$ of dimension two on $\mathbb{F}_{p^n}$ (with $p$ odd and $q=p^n$) represents another large family of finite dimensional Lie algebras on finite fields, where it is meaningful to study the subalgebra commutativity degree. This introduces our second main result.

\begin{theorem} \label{sdforsl2}
 If $p$ is an odd prime and $q=p^n$, then $\mathrm{sd}(\mathfrak{sl}_2(\mathbb{F}_q))$ has the same type of growth of $\mathrm{sd}(\mathfrak{h}(1))$ in Theorem \ref{heisenberg} replacing the role of $q$ with that of $p$, that is,   $$\mathrm{sd}(\mathfrak{sl}_2(\mathbb{F}_q))=\frac{3q^3+12q^2+16q+16}{{(q^2+2q+4)^2}}.$$
\end{theorem}

The final main result testifies that the theory of the subgroup commutativity degree of finite groups (see \cite{sr1, sr2, or, mt}) matches perfectly the notion in Definition \ref{sdl}. To this scope, we recall from \cite{Khukhro1998} that the \textit{Lazard Correspondence} is a fundamental  tool in the study of nilpotent groups, since it provides a method  for classifying finite $p$-groups in connection with appropriate Lie algebras. We follow \cite{Khukhro1998} for the terminology regarding nilpotent Lie algebras, nilpotent $p$-groups and Lazard Correspondence. In particular,  D'Elb\'ee and others \cite[\S 2.3]{lazard} present a short review on the Lazard Correspondence which we report here; for instance,  suppose that a finite $p$-group $P$ has nilpotency class $c(P)=c$ (see definitions in \cite{Khukhro1998, rs} ) and that $c<p$, where $p$ is an odd prime. If $P$ has exponent $p$, then we may define $L_P$ to be an algebraic structure with same elements of $P$ but endowed with  operations of sum
\[g+_{L_P} h=h_1(g, h)=g h[g, h]^{-\frac{1}{2}}[g, [g, h]]^{-\frac{1}{12}}[h, [g, h]]^{\frac{1}{12}} \ldots  \] and of multiplication
\[
[g, h]_{L_P}=h_2(g, h)=[g, h][g, [g, h]]^{\frac{1}{2}}[h, [g, h]]^{\frac{1}{2}} \ldots \]
where the brackets on the right denote the group commutators in $P$ and the powers $1/12, -1/12$ and so on originate from  power series expansions, which are available in \cite[Chapters 9, 10]{Khukhro1998}.  The two operations above  are called \textit{Inverse Baker-Campbell-Hausdorff Formulas} in \cite[Lemma 10.7] {Khukhro1998}, since they allow to put the structure of a differential manifold on a Lie group.  Since the nilpotency class of $P$ is at most $c$, it turns out that both $h_1(g,h)$ and $h_2(g,h)$ are finite products of group commutators in $P$ and may be written as powers of commutators in the set of all  rational numbers $q \in \mathbb{Q}$ such that if $q=\frac{l}{m}$ is in reduced form then $\gcd(m, p)=1$ (these are the so-called $\mathbb{Q}$-\textit{powered groups} in \cite[Chapter 9]{Khukhro1998}). On the other hand  $P$ has  exponent $p$, hence the aobve two operations are well-defined on $L_P$ and produce a Lie algebra $L_P$.  Conversely, given a Lie algebra $L$ on $\mathbb{F}_p$ of nilpotency class at most $c$, one may define $P_L$ to be the structure with the same elements of $L$ but with  binary operation $$
a *_{P_L} b=H(a, b)=a+b+\frac{1}{2}[a, b]+\frac{1}{12}[a, [a, b]]-\frac{1}{12}[b, [a, b]]+\ldots
$$
which produces a finite $p$-group. Roughly speaking, above we have just written the \textit{Baker-Campbell-Hausdorff Formula}, where the brackets on the right-hand side are the Lie brackets of $L$. Again in the present situation $H(a,b)$ can be written as an appropriate finite linear combination (thus be viewed as an $\mathbb{F}_{p}$-linear combination of Lie monomials according to \cite{lazard}).

\begin{theorem} \label{lazardforp3} Suppose $L$ is a nilpotent finite dimensional Lie algebra on \(\mathbb{F}_{p}\) for any prime $p\ge 3$ with nilpotency class $c < p$. If $P$ is the corresponding $p$-group which appears via the Lazard Correspondence, then  $$\mathrm{sd}(L)= \frac{|\{(H,K) \in \mathrm{L}(P) \times \mathrm{L}(P) \mid HK=KH\}|}{|\mathrm{L}(P)|^2},$$ where $\mathrm{L}(P)$ is the lattice of all subgroups of $P$
and this is exactly the subgroup commutativity degree of $P$.  In other words, the subalgebra commutativity degree is transformed in the subgroup commutativity degree by the Lazard Correspondence. 
\end{theorem}

After a selection of results from the large literature on modularity conditions in Lie algebras in Section \ref{sec2}, we develop the main results concerning the subalgebra commutativity degree in Section \ref{sec3}, overlapping methods and ideas of \cite{sr1, sr2, or, mt} but showing also some limits which we have in the context of Lie algebras. Finally the main proofs are placed in Section \ref{sec4}.

\section{A short review on modular conditions for finite dimensional Lie algebras} \label{sec2}

The only  known examples of modular subalgebras, which are not quasi-ideals, are the one dimensional subalgebras of the three dimensional nonsplit simple Lie algebras and the standard maximal subalgebra of certain Hamiltonian Lie algebras. As indicated by Strade \cite{strade1}, a \textit{simple Lie algebra} is a Lie algeba whose ideals are only the trivial ideals.  On the other hand, we say that a Lie algebra $L$ splits if it contains a nontrivial ideal $I$ and a subalgebra $J$ such that $L=I +J$. When this doesn't happen, $L$ is said to be \textit{nonsplit}. The formal definition of Hamiltonian Lie algebra is quite technical, so we omit it and refer to \cite[\S 4.2]{strade1}, as well as the classification of Cartan of \textit{standard Lie algebras} and \textit{nonstandard Lie algebras}, see \cite{strade1}. Let's only note  that for the three dimensional nonsplit simple Lie algebras, we can find Lie algebras of the form   $$L=\langle e\rangle + \langle f\rangle  + \langle g\rangle$$ on $\mathbb{F}$ containing no pairs of elements $(\alpha, \beta) \in \mathbb{F} \times \mathbb{F}$ such
that $$\alpha^2 +\beta^2=-1  \  \ \mbox{and} \ \  [e,f]=g, [f,g]=e, [e,g]=f.$$ In this situation  $L$ is simple, and every one dimensional subalgebra of $L$ is maximal and modular in $L$, but not a quasi-ideal.
It is also clear that the lattice of all ideals $\mathcal{I}(L)$ of a Lie algebra $L$ is a modular sublattice in $ \mathcal{L}(L)$. The smallest nonmodular lattice is of course the pentagonal lattice $N_5=\{0, A,B,C, L\}$ with $B \subseteq C $  such that
$\langle B,A\cap C \rangle= \langle B, 0 \rangle=B < C=L\cap C= \langle B,A \rangle \cap C $
holds, contradicting the modular law. However each nonmodular lattice contains a copy of $N_5$ as a sublattice; this is a well known fact in topological algebra, see \cite{hhr, rs}.

\begin{definition} [See \cite {kolman1965}, Almost abelian Lie algebras]\label{almostabelian}
A  Lie algebra on $\mathbb{F}$ of dimension $n+1$ is called \textit{almost abelian}, if it has a basis $e_0, e_1, \cdots, e_n$ such that $[e_i,e_0]=e_i$ for $i\geq 1$ and $[e_i,e_j]=0$ for  $i, j\geq 1$. 
\end{definition}

One can see that an almost abelian Lie algebra $L$ is of the form of a vector space sum  $L=L^2 + \mathbb{F}u$, where the \textit{derived subalgebra} $L^2 =[L, L]$ is abelian and the adjoint map \[\mathrm{ad}_x : y \in L \mapsto \mathrm{ad}_x(y) \in L\] is of the form $\mathrm{ad}_x(y)=[x,y]=y $ for all $y\in L^2$. This  means that $L$ contains an abelian ideal of codimension one, on which it acts by scalar multiplications. Kolman \cite [Proposition 1.1]{kolman1965} showed that  any   almost abelian Lie algebra $L$ of dimension $n+1$ is solvable with  \textit{nilpotent radical} 
$$\mathrm{Nil}(L)=\langle I \in \mathcal{I}(L) \mid I \ \mbox{is nilpotent}\rangle$$
which is  abelian. Note that $\mathrm{Nil}(L)$ is the  largest nilpotent ideal of $L$. There are a series of further useful results on finite dimensional modular Lie algebras in \cite{kolman1965}.  Following Kolman and others \cite{bowman2004, elduque1, elduque2, kolman1965, lashi, david1981}, we say that  a finite dimensional Lie algebra $L$ on $\mathbb{F}$ is upper semi-modular, lower semi-modular, or modular, if so is the lattice $\mathcal{L}(L)$.      Let's see more precisely the meaning of this notions, according with \cite{boto,  gein, kolman1965}. 

\begin{definition}[See \cite{boto, gein, kolman1965}, Upper/lower semi-modular/modular Lie algebras] \label{uls}
We say that  $U$ is \textit{upper semi-modular} in a finite dimensional Lie algebra $L$ on $\mathbb{F}$,  if for every subalgebra $S$ of $L$ such that $U \cap L$ is maximal in $U$ and $S$, then $U$ and $S$ are maximal in $\langle U, S \rangle$. We say that $L$ is \textit{completely upper semi-modular} if every subalgebra of $L$ is upper semi-modular in $L$. We say that $U$ is \textit{upper modular} in $L$, if $U$ is maximal in $\langle U, S \rangle$ for every subalgebra $S$ of $L$ such that $S\cap L$ is maximal in $S$. We say that $L$ is \textit{completely upper modular} if every subalgebra of $L$ is upper modular in $L$. Dually,  one can give the notion of  $U$  \textit{lower semi-modular} in $L$,  
of $L$  \textit{completely lower semi-modular},  
of $U$ \textit{lower modular} in $L$, of $L$ \textit{completely lower modular}.
\end{definition}

Note  that a nilpotent finite dimensional Lie algebra $L$ on a field of any characteristic must be lower semi-modular.   Moreover from \cite{kolman1965} we have that a Lie algebra $L$ on a field of any characteristic is distributive if and only if $L$ is one dimensional.  Also the role of the maximal subalgebras is relevant in this discusion: if all maximal subalgebras of $L$ are one dimensional, then $L$ is two dimensional quasiabelian, or three dimensional  nonsplit simple from \cite[Lemma 5.8]{boto}.

\begin{example} Looking at \cite{elduque1, elduque2, strade1}, nonsplit simple finite dimensional Lie algebras are called  \textit{special simple Lie algebras}, and a prominent example is given by the nonsplit simple three dimensional  Lie algebra $\mathbb{A}_1$. In fact, one can check that among simple three dimensional Lie algebras the condition of being nonsplit is equivalent to that of being upper semi-modular. In particular,  $\mathbb{A}_1$ turns out to be  modular, that is, both upper semi-modular and lower semi-modular.\end{example}  

A first relevant classification appears in the works of Kolman \cite{kolman1965}, see below.

\begin{proposition}[See \cite{kolman1965}, Theorem 2.4] A Lie algebra $L$ is upper semi-modular if and only if $L$ is abelian, almost abelian, or special simple.
\end{proposition}

If $L$ is a finite dimensional  Lie algebra, then one can define the  notion of  \textit{solvable radical}  $$\mathrm{Sol}(L)=\langle I \in \mathcal{I}(L) \mid I \ \mbox{is solvable}\rangle$$
of $L$ and of course  $\mathrm{Nil}(L) \subseteq \mathrm{Sol}(L)$, but in general the inclusion can be proper. Also $\mathrm{Sol}(L)$ can be introduced more conveniently as the  largest solvable ideal of $L$. It happens that for a finite dimensional upper semi-modular Lie algebra $L$,  either $ L=\mathrm{Sol}(L)$, that  is, $L$ is solvable, or, if this doesn't happen, thn $L/\mathrm{Sol}(L)$ is a special simple Lie algebra. 
A \textit{Borel subalgebra}, usually denoted as \(\mathfrak{b}\), of a Lie algebra $L$ is defined as a maximal solvable subalgebra. 
For $\mathfrak{sl}_2(\mathbb{F}_q)$, all Borel subalgebras are conjugate and two dimensional. An element $x \in L$ is called nilpotent if $\operatorname{ad} x$ is a nilpotent linear map, and semisimple if $\operatorname{ad} x$ is a semisimple linear map. In $\mathfrak{sl}_2(\mathbb{F}_q)$, a semisimple element is split if its characteristic polynomial splits over $\mathbb{F}_q$, and nonsplit otherwise.

\begin{remark}Note that a classical result of Cartan \cite{strade1} on the classification of finite dimensional Lie algebras shows that we  have a decomposition of the form $L= H + [L,L]$, where $H$ is the Cartan subalgebra of $L$. In particular, we may use Cartan's Theorem and  \cite [Theorem 2.3] {kolman1965}, in order to find that  solvable upper semi-modular Lie algebras are either  abelian or almost abelian.  We can say more: abelian, almost abelian and special simple Lie algebras are  modular. Therefore  upper semi-modular Lie algebras on a field of characteristic zero must be modular. \end{remark}

\begin{proposition}[See \cite{kolman1965}, Proposition 3.1]  A simple three dimensional Lie algebra $L$ on $\mathbb{F}$ must be of one (and only one) of the following types:
\begin{itemize}
    \item [(i).] If $L$ is nonsplit, then $L$ is lower semi-modular;
    \item [(ii).] If $\mathbb{F}$ is algebraically closed, then $L$ is lower semi-modular.
\end{itemize}
\end{proposition}

We may confirm the intuition that the concepts  in Definition \ref{uls} describe classes of finite dimensional Lie algebras which are close to be abelian Lie algebras.

\begin{proposition}[See \cite{boto}, Theorem 4.5] \label{classificationusm} Let $L$ be a finite dimensional Lie algebra on $\mathbb{F}$ of $\mathrm{char}(\mathbb{F})=0$. Then the following  conditions are equivalent:
 \begin{itemize}
     \item [(i).]  $L$ is completely upper modular; 
     \item [(ii).] $L$ is completely upper semimodula;
     \item [(iii).] $L$ is either abelian, or almost abelian or a three dimensional nonsplit simple Lie algebra.
 \end{itemize}
Also the following  conditions are equivalent:
  \begin{itemize}
     \item [(i).]  $L$ is completely lower modular;
     \item [(ii).] $L$ is completely lower semi-modular;
     \item [(iii).] $L=R \oplus S_1 \oplus \cdots \oplus S_k$, where $R=\mathrm{Sol}(L)$ is supersolvable and  $S_1, \ldots, S_k$  are non-isomorphic three dimensional simple Lie algebras.
 \end{itemize}
\end{proposition}

It is well known that a finite group is supersolvable if and only if every maximal subgroup has prime index. The corresponding case of Lie algebras was studied by Barnes  \cite[Theorem 7]{barnes1967}.  

\begin{definition}[See \cite{barnes1967}, Supersolvable Lie algebras] 
 A Lie algebra $L$ of dimension $n$ is supersolvable if there exists a chain of ideals \(0 = L_0 \subset L_1 \subset L_2 \subset \dots \subset L_n = L\) such that \(\dim(L_i) = i\) for each $i$.
\end{definition}

It turns out that supersolvalbe Lie algebras can be characterized (in analogy with the case of finite groups) by the fact that   every maximal subalgebra has codimension one. These are essentially the Lie algebras which appear in Proposition \ref{classificationusm} (iii). It is also useful to recall that if 

\medskip

-- $\mathcal{B}$ is the class of finite dimensional Lie algebras $L$ on a prescribed $\mathbb{F}$ and having the property that all modular pairs of subalgebras of $L$ are permutable ;
\medskip

--  $\mathcal{Q}$ is the class of those finite dimensional Lie algebras $L$ on $\mathbb{F}$ with the property that $A$ permutes with $B$ whenever $A$ and $B$ are subalgebras of $L$  maximal in $A\cup B$; 

\medskip

then the following inclusion is satisfied (see \cite[Corollary 1.3]{david1981}) $$\mathcal{B}\subset \mathcal{Q} .$$ 
Moreover, Towers \cite[Theorem 2.4]{david1981}  proved  that  the additional condition that   $[L,L]$ is nilpotent implies the following equivalent conditions for finite dimensional Lie algebras $L$ on $\mathbb{F}$:  
 $$L\in \mathcal{B} \Longleftrightarrow L\in \mathcal{Q} \Longleftrightarrow L \ \ \mbox{is supersolvable}.$$
 In particular,  the special linear Lie algebra $\mathfrak{sl}_2(\mathbb{C})$ is in   $  \mathcal{Q}$, but it is not supersolvable.

\begin{remark} \label{generalization}
Lashi \cite[Theorems 1 and 2]{lashi} gave an important generalization to infinite dimensional Lie algebras in connection with Proposition \ref{classificationusm}; in fact he showed that any modular Lie algebra $L$ on  $\mathbb{F}$ is either abelian, or almost abelian,  or $\mathbb{A}_1$, or  infinite dimensional nonsplit simple in which every proper subalgebra is one dimensional. This  introduces the Lie algebra $P_{\infty}$  in \cite{burde1, burde2, lashi, strade1}. \end{remark}

\section{The algebraic combinatorics of the subalgebra commutativity degree}\label{sec3}

Here we show some results which overlap the ideas and the techniques which are used for the subgroup commutativity degree of finite groups, focusing on differences which are automatically present in the context of Lie algebras.

\begin{proposition} We have
 $0<\mathrm{sd}(L)\leq 1$ for every finite dimensional Lie algebra $L$ on $\mathbb{F}_{p^n}$. Moreover $\mathrm{sd}(L) = 1$ iff all subalgebras of $L$ are permutable. In particular, this happens when $\mathcal{I}(L)=\mathcal{L}(L)$. \end{proposition}

 \begin{proof}
 Application of Definition \ref{sdl}.    
 \end{proof}

\begin{remark} $\mathbb{A}_1$ is modular, but has no quasi-ideals. However Amayo and Schwarz \cite{amsc} showed that a modular subalgebra $M$ of a Lie algebra $L$ permutable with a solvable subalgebra $S$ of $L$ is a quasi-ideal of $M + S$,  in particular $M$ is a quasi-ideal of $L$ if $L$ is solvable. \end{remark}

 \begin{example} \label{exam1} The Lie algebra
\(L=\langle x,y\ |\ [x,y]=x\rangle \) over the field \(\mathbb{F}_{p^n}\) is a non-nilpotent solvable modular Lie algebra.  In particular, consider this $L$ on \(\mathbb{F}_{2}\). The non-zero elements are $x,y$ and $x+y$ and this implies $\mathcal{L}(L)=\{ \{0\}, L, \langle x \rangle, \langle y \rangle, \langle x+y \rangle\}.$ Since  \([x,x]=0 \in \langle x\rangle \), $[y,x]=-[x,y]=-x=x $ belongs to $ \langle x\rangle$, hence \(\langle x\rangle \) is a one dimensional ideal.  But \([x,y]\notin \langle y\rangle \) and \([x,x+y]=x \notin \langle x+y\rangle \), and so both \(\langle y\rangle \) and \(\langle x+y\rangle \) are not ideals. However, \(\langle y\rangle \) permute with \(\langle x+y\rangle \), since $y+(x+y)=x+2y=x+0=x$  over $\mathbb{F}_{2}$ and 
\([\langle y\rangle,\langle x+y\rangle]\subseteq \langle y\rangle+\langle x+y\rangle = \{0,y,x,y+x\}=L\). Then $\langle y\rangle$ and $\langle x+y\rangle$ are quasi-ideals of $L$. Hence   $\mathrm{sd}(L)=1$, implying that $L$ is  non-nilpotent solvable modular. In addition  $\mathcal{L}(L)\simeq M_3$ is a diamond (i.e.: a lattice in which atoms and coatoms coincide, see \cite{rs}) and this is a minimal example of  non-distributive modular lattice. On the other hand,  the one dimensional subalgebras  of   $\mathfrak{sl}_2(\mathbb{F}_p)$ are atomic for any prime $p$, but are not of codimension one and so $\mathrm{sd}(\mathfrak{sl}_2(\mathbb{F}_p))<1$.  In fact, these Lie algebras are relatively far from being modular Lie algebras, see \cite{barnes1967, boto, elduque1}.  It can be checked also that   $\mathfrak{sl}_2(\mathbb{F}_p)$ has a subalgebra isomorphic to $N_5$, moreover Lemma \ref{sdandalmost} later on shows independently that the subalgebra commutativity degree of $L$ is one whenever $L$ is abelian or almost abelian.  \end{example}

It is more convenient to look at Definition \ref{sdl} from a more combinatorial perspective, in analogy to what has been done in \cite{sr1, sr2, or, mt}. As usual we consider $L$ finite dimensional Lie algebra on $\mathbb{F}_{p^n}$, but now let $\chi:\  \mathcal{L}(L)\times \mathcal{L}(L)\longrightarrow \{0,1\}$ be a function defined by 
\begin{equation} \label{charfun}
    \chi(A,B)=\left\{\begin{array}{lcl} 1,&\,\,& \mathrm{if} \   \ [A,B]\subseteq A+B , \\  
\\ 
 0 ,&\,\,& \mathrm{otherwise} . \\
 \end{array}\right.
\end{equation}
For an arbitrary subalgebra  $A$ of $L$, the set 
\begin{equation} \label{permutewithA}
        \mathcal{Q}_{\mathcal{L}(L)}(A) =\{B \in \mathcal{L}  (L)  \mid  \ [A, \  B] \subseteq A+B \}. 
\end{equation}
 containing the subalgebras of $L$ that permute with $A$ is the \textit{permutizer} of $A$ in $\mathcal{L}(L)$. Note that the set \begin{equation}\label{intersection}
\mathcal{Q}_{\mathcal{L}(L)}(\mathcal{L}(L)) =\{B \in \mathcal{L}(L) \ | \ [A,B] \subseteq A+B, \  \ \ \forall A \in \mathcal{L}(L)\} 
\end{equation}
 is precisely the family of quasi-ideals, testing against all subalgebras gives
the bracket condition against every vector subspace.




\begin{lemma}\label{combinatorialformulat} Given $L$ finite dimensional Lie algebra on $\mathbb{F}_{p^n}$, we have
\[\label{sdviacharfun}
    \mathrm{sd}(L)=\frac{1}{|\mathcal{L}(L)|^2} \underset{A, B\in \mathcal{L}(L) }{\sum} \chi (A,B).
\]    
\end{lemma}

\begin{proof} It is sufficient to apply definitions and observe that for   $A\in \mathcal{L}(L)$ we have
\[  
|\mathcal{Q}_{\mathcal{L}(L)}(A)|=\underset{B\in \mathcal{L}(L) }{\sum} \chi(A,B).
\]
\end{proof}
In general, if $I \in \mathcal{I}(L)$, then  $[I, H]\subseteq I \subseteq I+H$ for every  $H \in \mathcal{L}(L)$ and so $H$ is a quasi-ideal in $L$.  However, not every maximal subalgebra of $L$ is necessarily a quasi-ideal in $L$.

\begin{lemma} \label{codimone1}
Let $L$ be a  finite dimensional Lie algebra on $\mathbb{F}_p$ such that every maximal subalgebras of $L$ are of codimension one and atomic. Then 
$\mathrm{sd}(L)=1$. 
\end{lemma}
\begin{proof}
In the present situation,  if  $M \in \mathcal{L}(L)$ is of codimension one, then $L=M+\lambda x$ for $x\in L \backslash M $ and $\lambda \in \mathbb{F}_p$, and $[M, \langle x\rangle] \subseteq M+ \langle x\rangle$. This implies $M$ is a quasi-ideal in $L$, so  $\mathrm{sd}(L)=1$.
\end{proof}

\begin{lemma} \label{sdandalmost}
Let $L$ be an almost abelian Lie algebra of  dimension  over $\mathbb{F}_p$. Then $\mathrm{sd}(L)=1$. \end{lemma}
\begin{proof}From Definition \ref{almostabelian}, since $L$ has a basis $\{e_0, e_1, \cdots, e_n\}$ with $[e_i,e_0]=e_i$ for $i\geq 1$ and $[e_i,e_j]=0$ for $i, j\geq 1$, for arbitrary subspaces $A$ and $B$ of $L$ with $a=a_0e_0+a_1e_1+\cdots+a_1e_1\in A$ and $b=b_0e_0+b_1e_1+\cdots+b_1e_1\in B$ we have that $[a,b]=b_0a-a_0b\in A$, and $[a,b]=b_0a-a_0b\in B$. Therefore $[a,b]=b_0a-a_0b\in A+B$. This means $[A,B] \subseteq A+B$, or equivalently $\mathcal{L}(L)= \mathcal{Q}_{\mathcal{L}(L)}(\mathcal{L}(L))$. Hence the result follows. \end{proof}

We may adapt some arguments from \cite{or, mt}  via Lemmas \ref{combinatorialformulat}, \ref{codimone1} and \ref{sdandalmost}.

\begin{proposition} \label{boundintermsofideal}
Let $I$ be an arbitrary ideal of a finite dimensional Lie algebra ${L}$ on $\mathbb{F}_p$. Then
 \[
|\mathcal{L}(L)|^2 \ \     \operatorname{sd}(L) \geqslant \left(|\mathcal{L}(I)|+\left|\mathcal{L}\left(\frac{L}{I}\right)\right|-1\right)^2
+\Big(\operatorname{sd}(I)-1\Big)|\mathcal{L}(I)|^2+\left(\operatorname{sd}\left(\frac{L}{I}\right)-1\right)\left|\mathcal{L}\left(\frac{L}{I}\right)\right|^2.\]
\end{proposition}

\begin{proof}
  Consider the sets $\mathcal{D}_1=\{A\in \mathcal{L}(L)| I\subseteq A \}$  and  $\mathcal{D}_2=\{A\in \mathcal{L}(L)| A \subset I \}$.  Since $\mathcal{D}_1 \cup \mathcal{D}_2 \subseteq \mathcal{L}(L) $, then by \eqref{charfun} we have that 
  $$\mathrm{sd}(L)\geq \frac{1}{|\mathcal{L}(L)|^2}\underset{A,B\in \mathcal{D}_1 \cup \mathcal{D}_2}{\sum}{\chi(A,B)}=\frac{1}{|\mathcal{L}(L)|^2} \Bigg(\underset{A,B\in \mathcal{D}_1 }{\sum}{\chi(A,B)}+
  \underset{A,B\in  \mathcal{D}_2}{\sum}{\chi(A,B)}+
  2\underset{A\in \mathcal{D}_1 }  {\sum}  \underset{B\in \mathcal{D}_2 }  {\sum}{\chi(A,B)}
  \Bigg).$$
  Now, 
  $$\underset{A,B\in \mathcal{D}_1 }{\sum}{\chi(A,B)}= \operatorname{sd}\left(\frac{L}{I}\right)\left|\mathcal{L}\left(\frac{L}{I}\right)\right|^2,$$
  $$ \underset{A,B\in \mathcal{D}_2 }{\sum}{\chi(A,B)}= \underset{A,B\in \mathcal{D}_2 \cup \{I\}}{\sum}{\chi(A,B)} -2\underset{A\in \mathcal{D}_2 \cup \{I\} }{\sum}{\chi(A,B)}+1 \ \ \mbox{and} $$
  $$2\underset{A\in \mathcal{D}_1 }  {\sum}  \underset{B\in \mathcal{D}_2 }  {\sum}\chi(A,B)= 2 \ |\mathcal{D}_1| \ |\mathcal{D}_2| = 2 \ |\mathcal{L}(L/I)| \ \Big( |\mathcal{L}(I)|-1\Big).$$ 
  Therefore, the result follows by summing up all these inequalities. 
\end{proof}

\begin{corollary}
    Let $L$ be a finite dimensional Metabelian Lie algebra over $\mathbb{F}_p$ with both ideal $I$ and $L/I$ are abelian. Then 
    $$\mathrm{sd}(L)\geq \Bigg( \frac{|\mathcal{L}(I)|+ |\mathcal{L}(L/I)|-1}{|\mathcal{L}(L)|}\Bigg)^2.$$
\end{corollary}

\begin{proof}
Assume $L$ be a finite dimensional metabelian lie algebra on $\mathbb{F}_p$. It therefore possesses an ideal $I$ of $L$, and let $L/I$ and $I$ are both abelian. This implies $\mathrm{sd}(I)=\mathrm{sd}(L/I)=1$ and the result follows from Proposition \ref{boundintermsofideal}.
\end{proof}

\begin{corollary} \label{codimone}
    Let $L$ be a finite dimensional Lie algebra on $\mathbb{F}_p$ having an ideal $I$ of codimension one. Then 
    $$\mathrm{sd}(L)\geq \frac{1}{|\mathcal{L}(L)|^2}\Big(\mathrm{sd}(I) |\mathcal{L}(I)|^2+2|\mathcal{L}(I)|+1 \Big).$$  
\end{corollary}

\begin{proof}
Assume $I$  is of codimension one. Then $|\mathcal{L}(L/I)|=2$ and  $\mathrm{sd}(L/I) =1$.  Now apply Proposition \ref{boundintermsofideal}.
\end{proof}

\begin{corollary} Let $L$ be a solvable finite dimensional Lie algebra on $\mathbb{F}_p$ and let $\{0\}=L_0 \subset L_1 \subset \cdots \subset L_k=L$ be a series of subalgebras of $L$ such that all factor Lie algebras $L_i/L_{i-1}$ are one dimensional. Then 
    $$\operatorname{sd}(L) \geqslant \frac{1}{|\mathcal{L}(L)|^2}\left(2 \sum_{i=1}^k\left|\mathcal{L}\left(L_{i-1}\right)\right|+(k+1)\right).
$$
\end{corollary}
 \begin{proof}
From Corollary \ref{codimone}  we obtain
  $$\operatorname{sd}\left(L_i\right)\left|\mathcal{L}\left(L_i\right)\right|^2 \geqslant \operatorname{sd}\left(L_{i-1}\right)\left|\mathcal{L}\left(L_{i-1}\right)\right|^2+2\left|\mathcal{L}\left(L_{i-1}\right)\right|+1 \ \mbox{for} \ i=1,\cdots,k. $$ 
\[ \label{solvablebound} \Longrightarrow \ \operatorname{sd}\left(L_i\right)\left|\mathcal{L}\left(L_i\right)\right|^2 -
 \operatorname{sd}\left(L_{i-1}\right)\left|\mathcal{L}\left(L_{i-1}\right)\right|^2
 \geqslant
 2\left|\mathcal{L}\left(L_{i-1}\right)\right|+k \ \mbox{for} \ i=1,\cdots,k
\] and the result follows. \end{proof}

An idea of Farrokhi and others \cite{farrokhi1, farrokhi2}  can be adapted to the present situation. For a finite dimensional Lie algebra $L$ on $\mathbb{F}_{p^n}$ we define the \textit{factorization number} of $L$ by 
$$F_2(L)=|\{(A,B)\in \mathcal{L}(L) \times \mathcal{L}(L) \ | \ L= A + B \}|.$$ 
This denotes the number of all possible factorizations of $L$ as sum of two subalgebras $A$ and $B$. The subalgebra commutativity degree and factorization numbers are strongly connected.

\begin{proposition} \label{sdandf2} If $F_2(A)$ is the factorization number of an arbitrary subalgebra $A$ of a finite dimensional Lie algebra $L$ on $\mathbb{F}_{p^n}$, then
\[     \mathrm{sd}(L)=\frac{1}{|\mathcal{L}(L)|^2}\underset{A\in \mathcal{L}(L)}{\sum} F_2(A). \]
\end{proposition}

\begin{proof}
Since $F_2(A)=|{(B,D) \in \mathcal{L}(A)\times \mathcal{L}(A)}\  | \ [B,D] \subseteq B+D \ \mbox{and} \  B+D=A\}|$  and  $ |\mathcal{Q}_{\mathcal{L}(L)}(A)|=|\{B \in \mathcal{L}(L)  \mid A\cup B =A+B\ \mbox{and} \ [A, \  B] \subseteq A+B \}|$, we may conclude that
$$\underset{A\in \mathcal{L}(L)}{\sum} F_2(A)= \underset{A\in \mathcal{L}(L)}{\sum} |\mathcal{Q}_{\mathcal{L}(L)}(A)|.$$ Hence the result follows from Lemma \ref{combinatorialformulat}.
\end{proof}

Using the \textit{Möbius Inversion Formula}  for in Proposition \ref{sdandf2} we have 
\begin{equation} \label{sdandf2viamobius}
    F_2(L)= \underset{A\in \mathcal{L}(L)}{\sum} \mathrm{sd}(A) |\mathcal{L}(A)|^2 \mu(A,L),
\end{equation}
where the Möbius function  $\mu: \mathcal{L}(L) \times\mathcal{L}(L) \rightarrow \mathbb{Z}$ is recursively defined by: 
\begin{equation} \label{mobius}
    \sum_{C \in[B / A]} \mu(A, C)= \begin{cases}1, & A=B, \\ 
   \\ 
   0, & \text { otherwise } \end{cases}
\end{equation}
and   with a set
$[B / A]=\{C \in \mathcal{L}(L) \mid A \leq C \leq B\}$
for $A \leq B$,  called  \textit{interval} of $\mathcal{L}(L)$. Note that $[B / A]$ is a sublattice of $\mathcal{L}(L)$, see \cite{sr1, sr2, or, rs}.

\section{Proofs of the main results and final considerations}\label{sec4}

A  counting argument is sufficient for the proof of our first main result.

\begin{proof}[Proof of Theorem \ref{heisenberg}]
Since $\mathfrak{h}(1)$ has $p^3-1$ non-zero elements,  there are $p^3-1/p-1=p^2+p+1$   subalgebras of dimension one in $\mathfrak{h}(1)$, and among these \[\langle z\rangle=\{0, z, 2z, \cdots, (p-1)z\}\] is an ideal of $\mathfrak{h}(1)$. 
Assume that $A=\langle a \rangle$ and $B=\langle b \rangle$ are  one dimensional  subalgebras of $\mathfrak{h}(1)$, where $a=a_1x+a_2y+a_3z$ and $b=b_1x+b_2y+b_3z$ for $a_i,b_j\in \mathbb{F}_p$ with $i,j=1,2,3.$ 
Then 
\[[a,b]=[a_1x+a_2y+a_3z, \ b_1x+b_2y+b_3z]=(a_1b_2-a_2b_1)z\] implies that $A+B$ is a subalgebra of $\mathfrak{h}(1)$ if and only if $(a_1b_2-a_2b_1)z \in A+B$. This means that 
\[A  \ \mbox{permute with} \  B   \  \ \Longleftrightarrow \ \ (a_1b_2-a_2b_1)z \in A+B.\]  Therefore the condition  $[A, B]\subseteq A+B$ is satisfied if either $a_1b_2-a_2b_1=0$ or  if $z\in A+B$ in the case $a_1b_2-a_2b_1\neq 0$.  But, this happens only if both $A$ and $B$ are contained in the following sets  
\begin{itemize}
    \item [1).] $\{\langle  x+y \rangle, \langle x+y+z\rangle, \langle x+y+2z\rangle, \cdots , \langle x+y+(p-1)z\rangle \}$ or
    \item [2).] $\{\langle  x+2y \rangle, \langle x+2y+z\rangle, \langle x+2y+2z\rangle, \cdots , \langle x+2y+(p-1)z\rangle \}$ or 

$\vdots$

    \item [$p-1$).]  $\{\langle  x+(p-1)y \rangle, \langle x+(p-1)y+z\rangle, \langle x+(p-1)y+2z\rangle, \cdots , \langle x+(p-1)y+(p-1)z\rangle \}$  or
    \item [$p$).] $\{\langle  x \rangle, \langle x+z\rangle, \langle x+2z\rangle, \cdots , \langle x+(p-1)z\rangle \}$ or 
    \item [$p+1$).] $\{\langle  y \rangle, \langle y+z\rangle, \langle y+2z\rangle, \cdots , \langle y+(p-1)z\rangle \}$.  
\end{itemize}
We conclude that $\mathfrak{h}(1)$ has $p+1$ subalgebras of dimension two of  type $A+\langle z \rangle= B+A=B+\langle z \rangle $, where $A$ and $B$ are taken from the same class of sets listed above. Equivalently any chosen $A$ and $B$ from different classes of sets do not permute. Therefore we recognize in  $\mathcal{L}(\mathfrak{h}(1))$ the presence of   1 zero dimensional Lie subalgebra, of $p^2+p+1$  one dimensional Lie subalgebras, of $p+1$ two dimensional Lie subalgebras and of $1$ three dimensional Lie subalgebras. In particular,  $$|\mathcal{L}(\mathfrak{h}(1))|=1+(p^2+p+1)+(p+1)+1= {p^2+2p+4}. $$ 
Since each  of the two dimensional Lie subalgebras are of codimension one,  they are all maximal quasi-ideals. In addition,  $\langle z \rangle $,  $\{0\}$ and $\mathfrak{h}(1)$ are ideals of $\mathfrak{h}(1)$, then there are only ${p^2+p}$ nonquasi-ideals of which every $p$ of them permute each other. In other words  we find
$$|\mathcal{Q}_{\mathcal{L}(\mathfrak{h}(1))}(A)|=\left\{\begin{array}{lcl} p^2+2p+4,&\,\,& \mathrm{if} \ A \ \mbox{is a quasi-ideal}, \\  
\\ 
 2p+4,&\,\,& \mathrm{otherwise} . \\
 \end{array}\right.$$
Hence  
 $$\mathrm{sd}(\mathfrak{h}(1))= \frac{1}{(p^2+2p+4)^2} \Big (3 (p^2+2p+4) + (p^2+2p+4)(p+1)+  (2p+4) (p^2+p)\Big) $$ $$ =\frac{3p^3+12p^2+16p+16}{{(p^2+2p+4)^2}} . $$
\end{proof}




\begin{remark}Applying Theorem \ref{heisenberg} with $p=2$ and  $\mathbb{F}_2$, we get \[\mathrm{sd}(\mathfrak{h}(1))=\frac{5}{6},\] which corresponds to the subalgebra commutativity degree of $\mathfrak{h}(1))$. 
It is important to recall that in \cite{mt} the subgroup commutativity degree of the minimal non-abelian group of order 6, namely the symmetric group $S_3$, is \[\mathrm{sd}(S_3)=\frac{|\{(H,K) \in \mathrm{L}(S_3) \times \mathrm{L}(S_3) \mid HK=KH\}|}{|\mathrm{L}(S_3)|^2}=\frac{5}{6}\] and this means that there is no structural relationship between $S_3$, which is a finite group,  with $\mathfrak{h}(1))$, which is indeed a Heisenberg algebra,  because the number of permutable subgroups is generally  different   from the number of permutable subalgebras.  \end{remark}

Of course, we may proceed to compute $\mathrm{sd}(\mathfrak{h}(1))$ using Theorem \ref{heisenberg} for an odd prime $p$. However, let's compute $\mathrm{sd}(\mathfrak{h}(1)))$ explicitly below since this will be useful  later on.  

\begin{example} \label{exam3} Consider $\mathfrak{h}(1)=
\langle x,y,z|[x,y]=z\rangle $ on \(\mathbb{F}_{3}\). Then  $|\mathfrak{h}(1))|=27 $ and
$$\mathfrak{h}(1)=\{0,x,y,z,
2x,2y,2z, 
x+y, x+z, y+z,
2x+y, 2x+z, 2y+z,
x+2y, x+2z, y+2z, 
2x+2y, 2x+2z, $$
$$ 2y+2z,
x+y+z, 2x+2y+2z,
2x+y+z, x+2y+z, x+y+2z,
2x+2y+z, 2x+y+2z, x+2y+2z
\}.$$
We find by inspection in $\mathcal{L}(\mathfrak{h}(1))$ that there are $13$ one dimensional Lie subalgebras, namely 
\begin{multicols}{2}
\begin{itemize}
\item [1.] $\langle x \rangle=\{0, x,2x\}$
\item [2.] $\langle y \rangle=\{0, y, 2y\}$
\item [3.]   $\langle x+y \rangle=\{0, x+y,2x+2y\}$
\item [4.] $\langle x+z \rangle=\{0, x+z,2x+2z\}$
\item [5.] $\langle y+z \rangle=\{0, y+z,2y+2z\}$
\item [6.] $\langle x+2y \rangle=\{0, x+2y,2x+y\}$
\item [7.] $\langle x+2z \rangle=\{0, x+2z,2x+z\}$
\item [8.] $\langle y+2z \rangle=\{0, y+2z,2y+z\}$
\item [9.] $\langle x+y+z \rangle=\{0, x+y+z,2x+2y+2z\}$
\item [10.] $\langle 2x+y+z \rangle=\{0, 2x+y+z,x+2y+2z\}$
\item [11.] $\langle x+2y+z \rangle=\{0, x+2y+z,2x+y+2z\}$
    \item [12.] $\langle x+y+2z \rangle=\{0, x+y+2z,2x+2y+z\}$
    \item [13.] $\langle z \rangle= \{0, z,2z\}$
\end{itemize} 
\end{multicols}
Clearly $\langle z \rangle= \{0, z,2z\}$ is the only  one dimensional ideal of $\mathfrak{h}(1)$. Next we determine which of the two dimensional $A=\langle x \rangle \cup  \langle y \rangle=\{a_1x+a_2y+a_3z| a_1,a_2, a_3\in \mathbb{F}_{3}\}$ vector subspaces of $\mathfrak{h}(1)$ are subalgebras. Since $A$ will form a subalgebra if $\{0,z\} \subseteq A$,   it is easy to check that  
\begin{itemize}
\item [1.] $\langle x \rangle\cup \langle z \rangle= \langle x \rangle\cup \langle x+z \rangle=
\langle x \rangle\cup \langle x+2z \rangle$, 

\item [2.] $\langle y \rangle\cup \langle z \rangle= \langle y \rangle\cup \langle y+z \rangle=
\langle y \rangle\cup \langle y+2z \rangle$, 

\item [3.] $\langle x+y \rangle\cup \langle z \rangle= \langle x+y \rangle\cup \langle x+y+z \rangle=
\langle x+y \rangle\cup \langle x+y+2z \rangle$ and 

\item [4.] $\langle x+2y \rangle\cup \langle z \rangle= \langle x+2y \rangle\cup \langle x+2y+z \rangle=
\langle x+2y \rangle\cup \langle x+2y+2z \rangle$
are the two dimensional subalgebras, of course they are ideals of $\mathfrak{h}(1)$. 
\end{itemize}
One can see that $\{\langle x \rangle, \langle x+z \rangle,\langle x+2z \rangle\}$ is one of the set of permuting triples which are not quasi-ideals of $\mathfrak{h}(1)$. The same is true for the other classes. Therefore,  $$|\mathcal{Q}_{\mathcal{L}(\mathfrak{h}(1))}(A)|=\left\{\begin{array}{lcl} p^2+2p+4=19,&\,\,& \mathrm{if} \ A \ \mbox{is a quasi-ideal}, \\ 
\\ 
 2p+4=10,&\,\,& \mathrm{otherwise}.  \\
 \end{array}\right. 
 $$ Hence \[\mathrm{sd}(\mathfrak{h}(1))=\frac{253}{361}.\]
\end{example}

The logic of Example \ref{exam3} motivates the argument of the following proof.

\begin{proof}[Proof of Theorem \ref{sdforsl2}] 
We should remember that $q$ is a power of an odd prime $p$. Then, since 
 $\mathfrak{sl}_2(\mathbb{F}_q)=\Bigg \{ \begin{pmatrix}a&b \\ c&-a\end{pmatrix}\Big| a,b,c \in \mathbb{F}_q \Bigg \}$, it is easy to check that $ \Bigg \{e=\begin{pmatrix} 0 & 1 \\ 0 & 0 \end{pmatrix}, f=\begin{pmatrix} 0 & 0 \\ 1 & 0 \end{pmatrix}, h=\begin{pmatrix} 1 & 0 \\ 0 & -1 \end{pmatrix}\Bigg \}$ forms a basis for $\mathfrak{sl}_2(\mathbb{F}_q)$ with 
 \begin{equation} \label{bracket}
     [h, e] = 2e, \quad [h, f] = -2f , \quad [e, f] = h 
 \end{equation}
 and  $|\mathfrak{sl}_2(\mathbb{F}_q)|=q^3$.  Since every one dimensional subspace of $\mathfrak{sl}_2(\mathbb{F}_3)$ is a one dimensional Lie subalgebra, there are $\frac{q^3-1}{q-1}=q^2+q+1$ one dimensional Lie subalgebras such that
\begin{equation} \label{size of lines}
    q+1, \ \ \frac{q(q+1)}{2} \ \ \ \mbox{and}\ \ \ \frac{q(q-1)}{2}
\end{equation}
are  nilpotent, split semisimple  and  nonsplit semisimple,  respectively. Now,
using \eqref{bracket} we obtain that  $\langle e \rangle$, $\langle f \rangle$  and  $N_\mu=\langle h+\mu e -\mu^{-1}f\rangle$ for all $\mu \in\{1,2,\cdots, q-1\}$ are the one dimensional subalgebras coresponding to the nilpotent Lie subalgebras, and each  of them, along with the one dimensional subalgebras generated by split semisimple element, generates  two dimensional  subalgebras in $\mathfrak{sl}_2(\mathbb{F}_q)$. 
This means, if $S_i$ (with $ \ i \in\{1,2,\cdots, q-1\}$) is split semisimple, that
$$\mathfrak{b}_+=\langle \langle e \rangle, \langle h \rangle \rangle, \ \ \mathfrak{b}_-=\langle \langle f \rangle, \langle h \rangle \rangle, \ \
\mathfrak{b}_{\mu}= \langle N_\mu, S_i \rangle $$  are all the $q+1$ Borel subalgebras of $\mathfrak{sl}_2(\mathbb{F}_q)$, and one is nilpotent while $q$ split semisimple.  
This implies, including the trivial ideal and $\mathfrak{sl}_2(\mathbb{F}_q)$ itself, that
\begin{equation}\label{size of the lattice of sl2}
      |\mathcal{L}(\mathfrak{sl}_2(\mathbb{F}_q))|=2+(q^2+q+1)+(q+1)= q^2+2q+4 \ \ \mbox{and} \ \ |\mathcal{Q}_{\mathcal{L}(\mathfrak{sl}_2(\mathbb{F}_q))}(\mathcal{L}(\mathfrak{sl}_2(\mathbb{F}_q)))|= q+3.
\end{equation}
Now, we can apply \eqref{permutewithA} to obtain $|\mathcal{Q}_{\mathcal{L}(\mathfrak{sl}_2(\mathbb{F}_q))}(A)|$ for each $A \in \mathcal{L}(\mathfrak{sl}_2(\mathbb{F}_q))$.  So, we consider the following cases. 

\medskip
\textit{Case 1.} Assume $A \in \mathcal{Q}_{\mathcal{L}(\mathfrak{sl}_2(\mathbb{F}_q))}(\mathcal{L}(\mathfrak{sl}_2(\mathbb{F}_q)))$. 

\medskip
In this case it is obvious that $$\mathcal{Q}_{\mathcal{L}(\mathfrak{sl}_2(\mathbb{F}_q))}(A)=q^2+q+4.$$

\medskip
\textit{Case 2.} Let $A$ be  nilpotent subalgebra. 

\medskip
Since each nilpotent Lie algebra in $\mathfrak{sl}_2(\mathbb{F}_q)$ along with split semisimple are enough to generate a Borel subalgebra,  each of them can not be contained in two different Borel subalgebras. In fact,  the intersection of  two Borel  subalgebras is always a  one dimensional subalgebra which is generated by a semisimple element. This implies that each $q+1$ nilpotent Lie algebra of dimension one permutes with exactly $q+1$ one dimensional subalgebras contained in the same Borel subalgebra. In other words, \eqref{permutewithA}  implies 
$$\mathcal{Q}_{\mathcal{L}(\mathfrak{sl}_2(\mathbb{F}_q))}(A)=(q+1)+(q+3)=2q+4.$$ 

\medskip
\textit{Case 3.} Let $A$ be split semisimple subalgebra.

\medskip
In this case, each $A$ is the intersection of two different Borel subalgebras.  This means  $A$ permutes with $q+1$ subalgebras in the first Borel subalgebra and with the remaining $q$ subalgebras contained in the second Borel subalgebra. Thus,
$$\mathcal{Q}_{\mathcal{L}(\mathfrak{sl}_2(\mathbb{F}_q))}(A)=(q+1)+q=(2q+1)+(q+3)=3q+4.$$

\medskip
\textit{Case 4.} Let $A$ be  nonsplit semisimple subalgebra.

\medskip
Since there are exactly $\frac{q(q-1)}{2}$ maximal one dimensional subalgebras in $\mathfrak{sl}_2(\mathbb{F}_q)$, which do not generate a Borel subalbegra, then each of them do not permute with every other one dimensional subalgebra. This implies $$\mathcal{Q}_{\mathcal{L}(\mathfrak{sl}_2(\mathbb{F}_q))}(A)=q+4.$$ 
Therefore, by applying Lemma \ref{combinatorialformulat} with \eqref{size of lines} and \eqref{size of the lattice of sl2} we obtain
$$(q^2+2q+4)^2\mathrm{sd}(\mathfrak{sl}_2(\mathbb{F}_q)))=(q+1)(q^2+2q+4)+(q+3)(2q+4)+\Big(\frac{q(q+1)}{2}\Big)(3q+4)+\Big(\frac{q(q-1)}{2}\Big)(q+4).$$
This completes the proof.
\end{proof}

\begin{corollary} \label{f2ofsl2}
If $p$ is an odd prime and $q$ a prime power of $p$, then  
$F_2(\mathfrak{sl}_2(\mathbb{F}_q))=2q^3+5q^2+5q+7.$ In addition, $F_2(\mathfrak{h}(1))=2q^3+5q^2+5q+7.$
\end{corollary}

\begin{proof}
Applying \eqref{sdandf2viamobius}, we have 
$$F_2(\mathfrak{h}(1))
=\underset{A\in \mathcal{L}(\mathfrak{h}(1)) }{\sum} \mathrm{sd}(A) |\mathcal{L}(A)|^2 \mu (A, \mathfrak{h}(1)).  $$
Now from Theorem \ref{heisenberg} and  \eqref{mobius},
if $A=1$, then we have $$\mu(1,\mathfrak{h}(1))= -\Big(\mu(1,1)+(q^2+q+1)\mu(1,B)+(q+1)\mu(1,C)\Big)$$ where $B$ is a one dimensional and $C$ is a two dimensional subalgebras of  $\mathfrak{h}(1)$. 
But $\mu(1,1)=1$,
$\mu(1,B)=-(\mu(1,1))=-1$ and $\mu(1,C)= -\big(\mu(1,1)+(q+1)\mu(1,B)\big)=q$. This implies $\mu(1,\mathfrak{h}(1))=-(1+(q^2+q+1)(-1)+q(q+1)=0$.  Next if $A=B\neq \langle z\rangle$ with $\langle z\rangle$ is a quasi-ideal as shown in the proof of Theorem \ref{heisenberg}, then $\mu(A,\mathfrak{h}(1))=0$, if $A=\langle z\rangle$ then $\mu(\langle z\rangle,\mathfrak{h}(1))=-(\mu(\langle z\rangle, \langle z\rangle)+(q+1)\mu(\langle z\rangle,C))=q$, if $A=C$, then $\mu(\langle z\rangle,\mathfrak{h}(1))=-1$ and $\mu(\mathfrak{h}(1),\mathfrak{h}(1))=1$ for $A=\mathfrak{h}(1)$. Finally since $|\mathcal{L}(1)|=1$, $|\mathcal{L}(\langle z\rangle)|=|\mathcal{L}(B)|=2$,  $|\mathcal{L}(C)|=q+3$ , $|\mathcal{L}(\mathfrak{h}(1))|= q^2+2q+4$, $\mathrm{sd}(C)=1$ and from Theorem \ref{heisenberg} 
we obtain  
$$F_2(\mathfrak{h}(1))
=  (q+1)(q+3)^2(-1)+4q+ (3q^3+12q^2+16q+16)= 2q^3+5q^2+5q+7.$$
 To show the first part, assume $A=1$ and  let $D_1$, $D_2$ and $D_3$ be nilpotent, split semisimple and nonsplit semisimple Lie algebras, respectively. Then, using the counting arguments in the proof of Theorem \ref{sdforsl2} we have
$$\mu(1,\mathfrak{sl}_2(\mathbb{F}_q))= -\Big(\mu(1,1)+(q+1)\mu(1,D_1)+ \frac{q(q+1)}{2}\mu(1,D_2) + \frac{q(q-1)}{2}\mu(1,D_3) +
+(q+1)\mu(1,C)\Big)$$ where $C$ is a Borel subalgebras of  $\mathfrak{sl}_2(\mathbb{F}_q)$. 
But $\mu(1,1)=1$,
$\mu(1,D_1)=\mu(1,D_2)=\mu(1,D_3)=-(\mu(1,1))=-1$ and since each Borel subalgebra contains one $D_1$ and $q$ $D_1's$ we have
$\mu(1,C)= -\big(\mu(1,1)\mu(1,D_1)+q\mu(1,D_2)\big)=q$.
When $A=D_1$, then $\mu(D_1,\mathfrak{sl}_2(\mathbb{F}_q))=-(\mu(D_1,D_1)+ \mu(D_1,C))=0$. 
When $A=D_2$, then $\mu(D_2,\mathfrak{sl}_2(\mathbb{F}_q))=-(\mu(D_2,D_2)+ 2\mu(D_1,C))=1$. 
When $A=D_3$, then $\mu(D_3,\mathfrak{sl}_2(\mathbb{F}_q))=-(\mu(D_3,D_3)=-1$.
When $A=\mathfrak{sl}_2(\mathbb{F}_q)$, then $\mu(\mathfrak{sl}_2(\mathbb{F}_q),\mathfrak{sl}_2(\mathbb{F}_q))=1$. Therefore, since $|\mathcal{L}(1)|=1$, $|\mathcal{L}(D_1)|=|\mathcal{L}(D_2)|=|\mathcal{L}(D_3)=2$,  $|\mathcal{L}(C)|=q+3$ , $|\mathcal{L}(\mathfrak{sl}_2(\mathbb{F}_q))|= q^2+2q+4$, $\mathrm{sd}(C)=1$, then the result follows from Theorem \ref{sdforsl2}.
The proof follows completely. 
\end{proof}

\begin{example}For $m=1$ and an odd prime $p$, we  observe that $\mathfrak{sl}_2(\mathbb{F}_{p}) \not\simeq \mathfrak{h}(m) $.  However Theorems \ref{heisenberg} and  \ref{sdforsl2} show that $\mathrm{sd}(\mathfrak{sl}_2(\mathbb{F}_{p}))=\mathrm{sd}(\mathfrak{h}(1))$. This is to say that the subalgebra commutativity degree may be the same for non-isomorphic Lie algebras on the same field.
\end{example}

Before to prove our final main result, we may recall a few facts on the Lazard Correspondence.

\begin{lemma}[See \cite{Khukhro1998}, Mal'cev Correspondence, p.118]  \label{lazard1} Suppose to have a positive integer $c<p$ with $p$ odd prime. To every finite $p$-group $P$ of exponent $p$ and nilpotency class $\leq c$, the Lazard Correspondence associates a Lie algebra $L_P$ over $\mathbb{F}_p$ with the same underlying set $L_P=P$ and with operations $a+b=h_1(a, b)$ and $[a, b]=h_2(a, b)$, and $r a=a^r$ for every $r \in \mathbb{F}_p$. Conversely, for every Lie algebra $L$ over $\mathbb{F}_p$ of nilpotency class $\leq c$, there is a corresponding finite $p$-group $P_L$ of exponent $p$ with the same underlying set and group operation defined by $a * b=H(a, b)$ such that $a^r=r a$ for $r \in \mathbb{F}_p$. These operations are inverses to each other: as Lie algebras on $\mathbb{F}_p$, we have $L_{P_L}=L$ and additionally $P_{L_P}=P$ as finite $p$-groups.
\end{lemma}

Inspecting the proofs of Lemma \ref{lazard1}, one can see that a subset $ K \subseteq P$ is a subgroup of $P$ if and only if $ K \subseteq L$ is a Lie subalgebra. Moreover a subset $I$ is a normal subgroup of $P$ (always in the context of Lemma \ref{lazard1}) if and only if $I$ is an ideal of $L$. The nilpotency class of $P$ at the level of finite $p$-group coincides with the nilpotency class of $L$ at the level of nilpotent Lie algebra, the derived length of $P$ coincides with the derived length of $L$.

\begin{proof}[Proof of Theorem \ref{lazardforp3}]
Suppose $L$ is a nilpotent finite dimensional Lie algebra on \(\mathbb{F}_{p}\) for prime $p\ge 3$. 
Because $p$ is an odd prime and $L$ has nilpotency class $c\le p-1$, we may use the Lazard Correspondence to construct a group $P$ of order $p^{\mathrm{dim}\ (L)}$. Using  Lemma \ref{lazard1} and the fact that all ideals of $L$ become normal subgroups in $P$ and all non-ideal subalgebras of $L$ become non-normal subgroups in $P$. Now assume $A,B\in \mathcal{L}(L)$ correspond to $H,K\in \mathrm{L}(G)$ such that 
$C=\langle A, B\rangle  \ \ \mbox{and}\ \  J=\langle H, K\rangle.$
Then, $C$ and $J$ correspond and $|J|=p^{\mathrm{dim}\ (C)}$ and $$|HK|=\frac{|H||K|}{|H\cap K|}=p^{\mathrm{dim}\ (A)+\mathrm{dim}\ (B)-\mathrm{dim}\ (A\cap B)}=p^{\mathrm{dim}\ (A\cup B)}.$$
Then, $$HK=KH=J\in \mathrm{L}(G)\Longleftrightarrow  |HK|=|J| \Longleftrightarrow \mathrm{dim}\ (A+B) 
\Longleftrightarrow \mathrm{dim}\ (C) \Longleftrightarrow A+B=C\in \mathcal{L}(L).$$
Thus, Definition \ref{sdl} becomes the subgroup commutativity degree  in \cite{sr1, sr2, or, mt}.
\end{proof}

The following consequence can be checked with a direct computation.

\begin{corollary} \label{cor3}
For any odd prime $p$ and extra-special $p$-group  $E(p^3)$ of order $p^3$ and exponent $p$,   $$\mathrm{sd}(\mathfrak{h}(1)) =\mathrm{sd}(E(p^3)) = \frac{|\{(H,K)\in \mathrm{L}(E(p^3)) \times \mathrm{L}(E(p^3)) \mid HK=KH\}|}{|\mathrm{L}(E(p^3))|^2}.$$
\end{corollary}


\begin{thebibliography}{99}

\bibitem{amayo1976} R. K. Amayo, Quasi-ideals of Lie algebras I, \textit{Proc. London Math. Soc.} \textbf{33} (1976), 28--36.

\bibitem{amayo1976II} R. K. Amayo, Quasi-ideals of Lie algebras II, \textit{Proc. London Math. Soc.}  \textbf{33} (1976), 37–-64.

\bibitem{amsc} R. K. Amayo and J. Schwarz, Modularity in Lie algebras, \textit{Hiroshima Math. J.} \textbf{10} (1980), 311--322.


\bibitem{barnes1967} D.W. Barnes, On the cohomology of soluble Lie algebras, \textit{Math. Z.} \textbf{101} (1967), 343--349.

\bibitem{lazard} C. D'Elb\'ee, I. M\"uller, N. Ramsey and D. Siniora,  Model-theoretic properties of nilpotent groups and Lie algebras, \textit{J. Algebra} \textbf{662} (2025),  640--701.




\bibitem{boto} K. Bowman and D. A. Towers, Modularity conditions in Lie algebras, \textit{Hiroshima Math. J.} \textbf{19} (1989), 333--346.


\bibitem{bowman2004} K. Bowman, D Towers and V. Varea,  On Upper-modular subalgebras of a Lie algebra,  \textit{Proc.  Edinburgh Math. Soc.} (2004) \textbf{47}, 325–-337. 




\bibitem{burde1} D. Burde and W. A. Moens,  Semisimple decompositions of Lie algebras and prehomogeneous modules  \textit{J. Algebra}  \textbf{604} (2022), 664--681.

\bibitem{burde2}D. Burde, W. Moens and P. P\'aez-Guill\'an, Counterexamples to the Zassenhaus conjecture on simple modular Lie algebras, \textit{J. Algebra} \textbf{629} (2023), 21--37.



\bibitem{elduque1} A. Elduque, A note on modularity in Malcev algebras, \textit{Arch. Math.} \textbf{50} (1988), 424--428.

\bibitem{elduque2}  A. Elduque, On semi-modular Malcev algebras, \textit{Arch. Math.} \textbf{50} (1988), 328-- 336.


\bibitem{farrokhi1}M.Farrokhi and F.Saeedi,Factorization numbers of some finite groups,\textit{Glasgow Math.J.}\textbf{54}(2012),345--354.


\bibitem{farrokhi2} M.Farrokhi and F.Saeedi,Subgroup permutability degree of $PSL(2,p^n)$,\textit{Glasgow Math.J.}\textbf{55}(2013),581-590.


\bibitem{gein} A. G. Gein, Semi-modular Lie algebras, \textit{Siberian Math. J.} \textbf{17} (1976), 189--193.



\bibitem{hhr} W. Herfort, K.H. Hofmann and F.G. Russo, When is the sum of two closed subgroups closed in a locally compact abelian group?, \textit{Topology Appl.} \textbf{270} (2020), Article ID 106958. 



\bibitem{kegel1} O.H. Kegel, Produkte nilpotenter Gruppen, \textit{Arch. Math.} \textbf{12}  (1961), 90--93.

\bibitem{kegel2} O.H. Kegel, Untergruppenverbände endlicher Gruppen, die den Subnormalteilerverband echt enthalten, \textit{Arch. Math.} \textbf{30} (1978), 225--228. 


\bibitem{Khukhro1998} E. I. Khukhro, \textit{p-automorphisms of finite p-groups}, Cambridge University Press, Cambridge,  1998.

\bibitem{kolman1965} B. Kolman, Semi-modular Lie algebras, \textit{J. Sci. Hiroshima Univ. Ser. A-I} \textbf{29} (1965), 149--163.

\bibitem{lashi} A. A. Lashi, On Lie algebras with modular lattices of subalgebras, \textit{J. Algebra} \textbf{99} (1986),  80--88.

\bibitem{lenston}J.C. Lennox and S.E. Stonehewer, \textit{Subnormal subgroups of groups},  Clarendon Press, Oxford, 1987.  

\bibitem{sr1} S.K. Muhie, D. E. Otera and F.G. Russo, Factorization number and subgroup commutativity degree via spectral invariants, \textit{Comp. Appl. Math.} \textbf{42} (2023), Art. No. 132.


\bibitem{sr2} S.K. Muhie and F.G. Russo, The probability of commuting subgroups in arbitrary  lattices of subgroups, \textit{Int. J. Group Theory}  \textbf{10} (2021),  125--135.

\bibitem{or}D. E. Otera and F.G. Russo, Subgroup $S$-commutativity degrees of finite groups, \textit{Bull. Belg. Math. Soc. Simon Stevin} \textbf{19} (2012), 373--382.







\bibitem{rs}   R. Schmidt, \textit{Subgroup lattices of groups}, de Gruyter, Berlin, 1994.

\bibitem{strade1} H. Strade,  \textit{Simple Lie algebras over fields of positive characteristic I}, de Gruyter, Berlin, 2004.


\bibitem{mt}   M. T\v{a}rn\v{a}uceanu, Subgroup commutativity degrees of finite groups, \textit{J. Algebra} \textbf{321} (2009), 2508--2520.



\bibitem{david1987} D. A. Towers, Almost nilpotent Lie algebras, \textit{Glasgow Math. J.} \textbf{29} (1987), 7--11.

\bibitem{david1981} D. A. Towers, On Lie algebras in which modular pairs of subalgebras are permutable, \textit{J. Algebra} \textbf{68} (1981), 369--377.








\end{thebibliography}
\end{document}